\documentclass{amsart}
\usepackage[utf8]{inputenc}
\usepackage{amssymb}
\usepackage{amsmath}
\usepackage{amsthm}
\usepackage{bm}
\usepackage[dvipsnames]{xcolor}
\usepackage{tikz-cd}
\usepackage{chngcntr}

\newtheorem{theorem}{Theorem}[section]
\newtheorem{corollary}[theorem]{Corollary}
\newtheorem{lemma}[theorem]{Lemma}
\newtheorem{definition}[theorem]{Definition}
\newtheorem{proposition}[theorem]{Proposition}

\newtheorem{sublemma}{Lemma}[theorem]

\counterwithin{figure}{section}

\newcommand{\into}{\triangleright}
\newcommand{\within}{\triangleleft}

\title{Intermediate Intrinsic Density and Randomness}
\author{Justin Miller}
\thanks{The author would like to thank his advisor, Dr. Peter Cholak, for the advice, discussion, and support that made this project possible. The author would also like to thank Dr.\ Laurent Bienvenu, Dr.\ Denis Hirschfeldt, and an anonymous referee for helpful advice and resources provided in critique of earlier versions of this paper.}
\thanks{Partially supported by NSF-DMS-1854136}
\date{}

\begin{document}

\maketitle

\begin{abstract}
    Given any 1-random set $X$ and any $r\in(0,1)$, we construct a set of intrinsic density $r$ which is computable from $r\oplus X$. For almost all $r$, this set will be the first known example of an intrinsic density $r$ set which cannot compute any $r$-Bernoulli random set. To achieve this, we shall formalize the {\tt into} and {\tt within} noncomputable coding methods which work well with intrinsic density.
\end{abstract}

{\let\thefootnote\relax\footnote{{\textbf{MSC:} 03D32}}}
{\let\thefootnote\relax\footnote{\textbf{Keywords:} intrinsic density, stochasticity, randomness, computability}}

\section{Introduction}
We shall study the sets whose intrinsic density is in the open unit interval. By a result of Astor \cite{intrinsicdensity}, intrinsic density $r$ is equivalent to injection stochasticity $r$, which is the uniform variant of Kolmogorov-Loveland stochasticity. Our goal is to separate intrinsic density from randomness: for almost all $r$ in the open unit interval, we shall construct sets which have intrinsic density $r$, or injection stochasticity $r$, which cannot compute any set random with respect to the $r$-Bernoulli measure.
\\
\\
We briefly recall the notion of (asymptotic) density in the natural numbers:

\begin{definition}
Let $A\subseteq\omega$.
\begin{itemize}
    \item The density of $A$ at $n$ is $\rho_n(A)=\frac{|A\upharpoonright n|}{n}$, where $A\upharpoonright n=A\cap\{0,1,2,\dots,n-1\}$.  
    \item The upper density of $A$ is $\overline{\rho}(A)=\limsup_{n\to\infty} \rho_n(A)$. 
    \item The lower density of $A$ is $\underline{\rho}(A)=\liminf_{n\to\infty} \rho_n(A)$. 
    \item If $\overline{\rho}(A)=\underline{\rho}(A)=\alpha$, we call $\alpha$ the density of $A$ and denote it by $\rho(A)$.
\end{itemize}
\end{definition}

Astor \cite{intrinsicdensity} introduced intrinsic density, which requires that the asymptotic density remain fixed under any computable permutation.

\begin{definition}
\begin{itemize}
    \item The absolute upper density of $A$ is 
    \[\overline{P}(A)=\sup\{\overline{\rho}(\pi(A)):\pi\text{ a computable permutation}\}\]
    \item The absolute lower density of $A$ is \[\underline{P}(A)=\inf\{\underline{\rho}(\pi(A)):\pi\text{ a computable permutation}\}\]
    \item If $\overline{P}(A)=\underline{P}(A)=\alpha$, we call $\alpha$ the intrinsic density of $A$ and denote it by $P(A)$. In particular, $P(A)=\alpha$ if and only if $\rho(\pi(A))=\alpha$ for every computable permutation $\pi$.
\end{itemize}
\end{definition}

A natural question to ask is what reals in the unit interval can be achieved as the intrinsic density of some set. Sets of intrinsic density $0$ and $1$ are well-known. In the open unit interval, Bienvenu (personal communication) gave a straightforward probabilistic argument that shows the sets of intrinsic density $r$ have measure $1$ under the $r$-Bernoulli measure. In fact, something stronger is true.

\begin{definition}
\label{randomnessmeasure}
Let $0<r<1$ be a real number. The Bernoulli measure with parameter $r$, $\mu_r$, is the measure on Cantor space such that for any $\sigma\in2^{<\omega}$, 
\[\mu_r(\sigma)=r^{|\{n<|\sigma|:\sigma(n)=1\}|}(1-r)^{|\{n<|\sigma|:\sigma(n)=0\}|}\]
If $X$ is ML-random with respect to $\mu_r$, we say it is $r$-random.
\end{definition}

If we say 1-random, we are specifically referring to $\frac{1}{2}$-ML-randomness. For a review of randomness with respect to noncomputable measures, see Reimann and Slaman \cite{ReimannSlaman}.

\begin{proposition}
\label{randomdensity}
Let $r\in(0,1)$. If $X$ is $r$-random, then $X$ has intrinsic density $r$.
\end{proposition}

Standard arguments for $\frac{1}{2}$ can be generalized to $r$ to prove this. It is true for $r$-Schnorr randomness. See, for example, Nies \cite[Theorem 3.5.21]{Nies} for the proof in the case $r=\frac{1}{2}$.
\\
\\
Thus every real in the unit interval is achieved as the intrinsic density of some set. Furthermore, we will not be able to find an ML-random set which does not have intrinsic density as was done for computable randomness and MWC stochasticity by Ambos-Spies \cite{Ambosspies} and was done for Schnorr randomness and Church stochasticity by Wang \cite{wang}. However, it is still possible to find sets which have intrinsic density $r$ but are not $r$-random. We shall not only do this, but in fact provide examples which cannot even compute an $r$-random set. To achieve this, we shall next develop some tools which work well with both asymptotic and intrinsic density and allow us to change densities. Computable operations like the join can only preserve intrinsic density, and set operations such as the intersection and union behave unpredictably with respect to asymptotic density.
\\
\\
We shall follow the convention, unless otherwise stated, that capital English letters represent sets of natural numbers and the lowercase variant, indexed by a subscript of natural numbers, represents the elements of the set. As an example, if $E$ is the set of even numbers, then $e_n=2n$. Recall that the principal function for a set $A$, $p_A$, is defined via $p_A(n)=a_n$.
\\
\\
Using this representation, it is not hard to see the following characterization of upper and lower density:

\begin{lemma}
\label{technicallimit}
Let $A\subseteq\omega$ be $\{a_0<a_1<a_2<\dots\}$. Then
\begin{itemize}
    \item $\overline{\rho}(A)=\limsup_{n\to\infty}\frac{n}{a_n}$
    \item $\underline{\rho}(A)=\liminf_{n\to\infty}\frac{n}{a_n}$
\end{itemize}
\end{lemma}
\begin{proof}
    Note that if $A\upharpoonright n+1$ has a $0$ in the final bit, then 
    \[\rho_n(A)=\frac{|A\upharpoonright n|}{n}>\frac{|A\upharpoonright n|}{n+1}=\rho_{n+1}(A)\]
    Therefore, to compute the upper density it suffices to check only those numbers $n$ for which $A\upharpoonright n$ has a $1$ as its last bit. Those numbers are exactly $a_n+1$ by the definition of $a_n$, and $|A\upharpoonright a_n+1|=n+1$. Therefore $\{\frac{n+1}{a_n+1}\}_{n\in\omega}$ is a subsequence of $\{\rho_n(A)\}_{n\in\omega}$ which dominates the original sequence, so $\overline{\rho}(A)=\limsup_{n\to\infty}\rho_n(A)=\limsup_{n\to\infty}\frac{n+1}{a_n+1}$. Finally, $\limsup_{n\to\infty}\frac{n+1}{a_n+1}=\limsup_{n\to\infty}\frac{n}{a_n}$.
    \\
    \\
    Similarly, to compute the lower density it suffices to check only the numbers $n$ such that the final digit of $A\upharpoonright n$ is a $0$, but the final digit of $A\upharpoonright n+1$ is a $1$. (That is, if there is a consecutive block of zeroes in the characteristic function of $A$, we only need to check the density at the end of the block when computing lower density, as each intermediate point of the zero block has a higher density than the end.) These numbers are exactly $a_n$ by definition, and $|A\upharpoonright a_n|=n$. Therefore $\{\frac{n}{a_n}\}_{n\in\omega}$ is a subsequence of $\{\rho_n(A)\}_{n\in\omega}$ which is dominated by the original sequence, so $\underline{\rho}(A)=\liminf_{n\to\infty}\rho_n(A)=\liminf_{n\to\infty}\frac{n}{a_n}$.
\end{proof}

A critical proof technique will involve proving that two sets $A$ and $B$ cannot have different intrinsic densities by creating a computable permutation which sends $A$ to $B$ modulo a set of density zero. The following lemma shows that if we can do this, then the density of the image of $A$ is the same as the density of $B$, and therefore that they cannot have different intrinsic densities.

\begin{lemma}
\label{densityzero}
If $\overline{\rho}(H)=0$, then $\overline{\rho}(X\setminus H)=\overline{\rho}(X\cup H)=\overline{\rho}(X)$ and $\underline{\rho}(X\setminus H)=\underline{\rho}(X\cup H)=\underline{\rho}(X)$.
\end{lemma}

\begin{proof}
Notice that
\[\rho_n(X)=\rho_n(X\setminus H)+\rho_n(X\cap H)\]
By definition. Therefore
\[\overline{\rho}(X)=\limsup_{n\to\infty}\rho_n(X)=\limsup_{n\to\infty} \rho_n(X\setminus H)+\rho_n(X\cap H)\]
By subadditivity of the limit superior,
\[\overline{\rho}(X)\leq\limsup_{n\to\infty} \rho_n(X\setminus H)+\limsup_{n\to\infty} \rho_n(X\cap H)\]
As $\overline{\rho}(H)=0$ and $X\cap H\subseteq H$,
\[\overline{\rho}(X)\leq\limsup_{n\to\infty} \rho_n(X\setminus H)=\overline{\rho}(X\setminus H)\]
However, $\overline{\rho}(X\setminus H)\leq\overline{\rho}(X)$ because $X\setminus H\subseteq X$, so $\overline{\rho}(X)=\overline{\rho}(X\setminus H)$ as desired.
\\
\\
The argument for the union and the argument for lower density are functionally identical. (For the union we use $X\cup H$, $X$, and $H\setminus X$ in place of $X$, $X\setminus H$, and $X\cap H$ respectively.)
\end{proof}

Our strategy is to find some process which takes a set $A$ of intrinsic density $\alpha$ and a set $B$ of intrinsic density $\beta$ and codes $B$ and $A$ in such a way that we are left with a set which has new intrinsic density obtained as some function of $\alpha$ and $\beta$. The following lemma will be useful for computing intrinsic densities.

\begin{lemma}
\label{permutationlemma}
Let $f_0,f_1,\dots,f_k$ be a finite collection of injective computable functions and let $C$ be an infinite computable set. Then there is an infinite computable set $H\subseteq C$ such that $\overline{\rho}(f_i(H))=0$ for all $i$.
\end{lemma}

\begin{proof}
Let $h_0=c_0$. Then given $h_n$, define $h_{n+1}$ to be the least element $c$ of $C$ with $f_i(c)\geq h_n!$ for all $i$. Set $H=\{h_0<h_1<h_2<\dots\}$. Then $\overline{\rho}(f_i(H))=0$ for all $i$ because $|f_i(H)\upharpoonright n|\leq |\{n!:n\in\omega\}\upharpoonright n|$.
\end{proof}

We shall see in Section 2 that we cannot hope for such a process to be computable in a way that allows us to recover the original sets. Instead we shall use the following noncomputable coding methods. They are natural and have been used informally by others such as Jockusch and Astor.

\begin{definition}
\label{intowithin}
    Let $A$ and $B$ be sets of natural numbers.
    \begin{itemize}
        \item  $B\into A$, or $B$ into $A$, is
        \[\{a_{b_0}<a_{b_1}<a_{b_2}<\dots\}\]
        That is, $B\into A$ is the subset of $A$ obtained by taking the ``$B$-th elements of $A$.'' 
        \item $B\within A$, or $B$ within $A$, is
        \[\{n:a_n\in B\}\]
        That is, $B\within A$ is the set $X$ such that $X\into A=A\cap B$.
    \end{itemize}
\end{definition}

With $B\into A$, we are thinking of $A$ as a copy of $\omega$ as a well-order and $B\into A$ is the subset corresponding to $B$ under the order preserving isomorphism between $A$ and $\omega$. We make a few elementary observations to illustrate how these operations behave:
\begin{itemize}
    \item For all $A$, $A=A\into \omega=\omega\into A=A\within\omega$.
    \item For all $A$ and $B$ and any $i$, $a_i$ is either in $B$ or $\overline{B}$. Therefore $i$ is either in $B\within A$ or $\overline{B}\within A$ respectively, so $(B\within A)\sqcup(\overline{B}\within A)=\omega$.
    \item If $B\cap C=\emptyset$, then $(B\into A)\cap(C\into A)=\emptyset$. Furthermore, $A=(X\into A)\sqcup(\overline{X}\into A)$.
    \item $\into$ is associative, i.e.\ $B\into(A\into C)=(B\into A)\into C$: By definition, $(A\into C)=\{c_{a_0}<c_{a_1}<c_{a_2}<\dots\}$ and thus 
        \[B\into (A\into C)=\{c_{a_{b_0}}<c_{a_{b_1}}<c_{a_{b_2}}<\dots\}\]
        Similarly, $(B\into A)=\{a_{b_0}<a_{b_1}<a_{b_2}<\dots\}$, and therefore by definition 
        \[(B\into A)\into C=\{c_{a_{b_0}}<c_{a_{b_1}}<c_{a_{b_2}}<\dots\}\]
    \item $\within$ is not associative: Consider the set of evens $E$, the set of odds $O$, and the set $N$ of evens which are not multiples of $4$. Then
    \[(O\within N)\within E=\emptyset\within N=\emptyset\]
    However,
    \[O\within(N\within E)=O\within O=\omega\]
    \item $\into$ and $\within$ do not associate with each other in general:
    \[B\into(A\within(B\into A))=B\into\omega=B\]
    but
    \[(B\into A)\within(B\into A)=\omega\]
    Similarly, $B\within(A\into B)=\omega$, but $(B\within A)\into B$ is a subset of $B$.
\end{itemize}

In Section 2, we shall construct examples of sets which have intrinsic density $r$ but are not $r$-random. However, all of these examples will trivially compute $r$-random sets, so we shall then turn our attention to constructing examples which cannot. In Section 3 we shall prove our main technical theorems which allow us to complete the construction in Section 4.

\section{Nonrandom sets of intrinsic density $r$}

The following theorem allows us to easily generate sets of intrinsic density $r$  which are not $r$-random from Proposition \ref{randomdensity}.

\begin{theorem}
\label{join}
$P(A\oplus B)=r$ if and only if $P(A)=P(B)=r$.
\end{theorem} 

\begin{proof}
We shall first argue the forward direction via contrapositive. We proceed by showing that, for any given computable permutation $\pi$, there are computable permutations which send $A\oplus B$ to $\pi(A)$ modulo a set of density $0$, and similarly for $\pi(B)$. Then the upper (and lower) density of $A\oplus B$ under these permutations will match that of $\pi(A)$ and $\pi(B)$ respectively. Therefore if $P(A)\neq P(B)$, the density of $A\oplus B$ is not invariant under computable permutation.
\\
\\
Let $F=\{n!:n\in\omega\}$ and $G=\overline{F}$. For any fixed computable permutation $\pi$, there is another computable permutation $\hat{\pi}$ defined via enumerating the odds onto the factorials in order and enumerating the evens onto the nonfactorials according to the ordering induced by $\pi$. That is, $\hat{\pi}(2n+1)=f_n$ and $\hat{\pi}(2n)=g_{\pi(n)}$.
\\
\\
Then as $F$ has density $0$, Lemma \ref{densityzero} shows
\[\overline{\rho}(\hat{\pi}(A\oplus B))=\overline{\rho}(\hat{\pi}(A\oplus B)\setminus F)\]
As the image of the odds under $\hat{\pi}$ is a subset of $F$, 
\[\hat{\pi}(A\oplus B)\setminus F=\hat{\pi}(A\oplus\emptyset)\]
and
\[\overline{\rho}(\hat{\pi}(A\oplus B))=\overline{\rho}(\hat{\pi}(A\oplus \emptyset))\]
Notice that $\hat{\pi}(A\oplus\emptyset)$ is just $\pi(A)$ with each element $n$ increased by $|F\upharpoonright n|$. Thus 
\[\rho_n(\pi(A))\geq\rho_n(\hat{\pi}(A\oplus\emptyset))\geq\frac{|\pi(A)\upharpoonright n|-|F\upharpoonright n|}{n}\]
As $F$ is the factorials, the final expression tends to $\rho_n(\pi(A))$ in the limit, so we see that
\[\overline{\rho}(\hat{\pi}(A\oplus\emptyset))=\overline{\rho}(\pi(A))\]
and
\[\overline{\rho}(\hat{\pi}(A\oplus B))=\overline{\rho}(\hat{\pi}(A\oplus\emptyset))=\overline{\rho}(\pi(A))\]
$\underline{\rho}(\hat{\pi}(A\oplus B))=\underline{\rho}(\pi(A))$ by a nearly identical argument.
\\
\\
In particular, $\overline{P}(A\oplus B)\geq \overline{P}(A)$ and $\underline{P}(A\oplus B)\leq \underline{P}(A)$ because we are taking the limit superior and inferior over all computable permutations, of which $\hat{\pi}$ is but one. Reversing the use of the evens and the odds in the definition of $\hat{\pi}$, we get that the same is true for $B$ in place of $A$, so $\underline{P}(A\oplus B)\leq min(\underline{P}(A),\underline{P}(B))$ and $\overline{P}(A\oplus B)\geq max(\overline{P}(A),\overline{P}(B))$. Therefore if $P(A)\neq P(B)$, $\underline{P}(A\oplus B)\neq\overline{P}(A\oplus B)$ as desired.
\\
\\
We now prove the reverse direction using a technical lemma.
    
\begin{sublemma}
\label{joinlemma}
    Let $\pi$ be a computable permutation and let $P(A)=P(B)=r$. Then
    \[\rho(\pi(A\oplus B)\within\pi(E))=\rho(\pi(A\oplus B)\within\pi(O))=r\]
\end{sublemma}
    
\begin{proof}
    Let $h:\pi(E)\to\omega$ send the $n$-th element of $\pi(E)$ to $n$ (the inverse of the principal function), and let $d:\omega\to E$ be defined via $d(n)=2n$. Then notice that $d(A)=A\oplus\emptyset$. Furthermore, observe that for any $X\subseteq\pi(E)$, $h(X)=X\within \pi(E)$ by the definition of $h$ and the {\tt within} operation. Therefore
    \[h(\pi(d(A)))=h(\pi(A\oplus\emptyset))=\pi(A\oplus\emptyset)\within\pi(E)\]
    As $\pi(A\oplus B)\cap \pi(E)\subseteq \pi(A\oplus \emptyset)$,
    \[\pi(A\oplus\emptyset)\within\pi(E)=\pi(A\oplus B)\within\pi(E)\]
    Thus $h(\pi(d(A)))=\pi(A\oplus B)\within\pi(E)$. We shall now massage $h$ and $d$ into permutations which preserve the relevant densities.
    \\
    \\
    By Lemma \ref{permutationlemma}, there is a computable set $H\subseteq \pi(E)$ with $\overline{\rho}(h(H))=0$. Now define the computable permutation $\pi_h$ via $\pi_h(n)=h(n)$ for $n\in\pi(E)\setminus H$, and have $\pi_h$ enumerate $\pi(O)\sqcup H$ onto $h(H)$ in order. Similarly, define the computable permutation $\pi_d$ via $\pi_d(n)=d(n)$ for $n\in\omega\setminus d^{-1}(\pi^{-1}(H))$, and have $\pi_d$ enumerate $d^{-1}(\pi^{-1}(H))$ onto $O\sqcup \pi^{-1}(H)$. 
    \\
    \\
    As $\pi_d$ agrees with $d$ on $\overline{d^{-1}(\pi^{-1}(H))}$, we now see that
    \[\pi_d(A\setminus \pi_d^{-1}(\pi^{-1}(H)))=(A\oplus\emptyset)\setminus \pi^{-1}(H)\]
    Furthermore, applying $\pi$ shows that
    \[\pi(\pi_d(A\setminus \pi_d^{-1}(\pi^{-1}(H))))=\pi((A\oplus\emptyset)\setminus \pi^{-1}(H))=\pi(A\oplus\emptyset)\setminus H\]
    As $\pi_h$ agrees with $h$ on $\pi(E)\setminus H$ and $h(\pi(A\oplus\emptyset))=\pi(A\oplus B)\within\pi(E)$, we have
    \[\pi_h(\pi(A\oplus\emptyset)\setminus H)=(\pi(A\oplus B)\within\pi(E))\setminus h(H)\]
    Therefore $(\pi(A\oplus B)\within\pi(E))\setminus h(H)\subseteq\pi_h(\pi(\pi_d(A)))$ and     $\pi_h(\pi(\pi_d(A)))\subseteq (\pi(A\oplus B)\within\pi(E))\cup h(H)$.
    \\
    \\
    By choice of $H$, $\overline{\rho}(h(H))=0$, so Lemma \ref{densityzero} shows that
    \[\overline{\rho}(\pi_h(\pi(\pi_d(A))))=\overline{\rho}((\pi(A\oplus B)\within\pi(E))\setminus h(H))=\overline{\rho}(\pi(A\oplus B)\within\pi(E))\]
    and
    \[\underline{\rho}(\pi_h(\pi(\pi_d(A))))=\underline{\rho}((\pi(A\oplus B)\within\pi(E))\setminus    h(H))=\underline{\rho}(\pi(A\oplus B)\within\pi(E))\]
    Therefore, as $P(A)=r$ and $\pi_h\circ\pi\circ\pi_d$ is a computable permutation, 
    \[\rho(\pi(A\oplus B)\within\pi(E))=r\]
    A nearly identical argument with $O$ in place of $E$ and $B$ in place of $A$ shows similarly that
    \[\rho(\pi(A\oplus B)\within\pi(O))=r\]
\end{proof}
    
We shall now show that this implies that $\rho(\pi(A\oplus B))=r$.
Consider $\rho_n(\pi(A\oplus B))$. By definition,
\[\rho_n(\pi(A\oplus B))=\frac{|\pi(A\oplus B)\upharpoonright n|}{n}\]
As $\omega=\pi(E)\sqcup\pi(O)$,
\[\frac{|\pi(A\oplus B)\upharpoonright n|}{n}=\frac{|\pi(A\oplus B)\cap\pi(E)\upharpoonright n|+|\pi(A\oplus B)\cap\pi(O)\upharpoonright n|}{n}\]
The latter expression can be rewritten as
\[\frac{|\pi(E)\upharpoonright n|}{|\pi(E)\upharpoonright n|}\cdot\frac{|\pi(A\oplus B)\cap\pi(E)\upharpoonright n|}{n}+\frac{|\pi(O)\upharpoonright n|}{|\pi(O)\upharpoonright n|}\cdot\frac{|\pi(A\oplus B)\cap\pi(O)\upharpoonright n|}{n}\]
Let $m$ be the largest number such that the $m$-th element of $\pi(E)$ is less than $n$, and let $k$ be the analogous number for $\pi(O)$. Now notice that
\[\frac{|\pi(A\oplus B)\cap\pi(E)\upharpoonright n|}{|\pi(E)\upharpoonright n|}=\rho_m(\pi(A\oplus B)\within\pi(E))\]
and
\[\frac{|\pi(A\oplus B)\cap\pi(O)\upharpoonright n|}{|\pi(O)\upharpoonright n|}=\rho_k(\pi(A\oplus B)\within\pi(O))\]
by the definition of the {\tt within} operation. Therefore, we can rewrite $\rho_n(\pi(A\oplus B))$ as  \[\rho_m(\pi(A\oplus B)\within\pi(E))\cdot\rho_n(\pi(E))+\rho_k(\pi(A\oplus B)\within\pi(O))\cdot\rho_n(\pi(O))\]
Using the fact that $\rho_n(\pi(E))+\rho_n(\pi(O))=1$,
\[\rho_n(\pi(A\oplus B))=\rho_m(\pi(A\oplus B)\within\pi(E))\cdot\rho_n(\pi(E))+\rho_k(\pi(A\oplus B)\within\pi(O))\cdot(1-\rho_n(\pi(E)))\]
Rearranging, this is equal to
\[\rho_k(\pi(A\oplus B)\within\pi(O))+\rho_n(\pi(E))\cdot(\rho_m(\pi(A\oplus B)\within\pi(E))-\rho_k(\pi(A\oplus B)\within\pi(O)))\]
Taking the limit as $n$ goes to infinity, $m$ and $k$ both go to infinity. Thus 
\[\rho_m(\pi(A\oplus B)\within\pi(E))-\rho_k(\pi(A\oplus B)\within\pi(O))\]
goes to $0$ by Lemma \ref{joinlemma}. As $\rho_n(\pi(E))$ is bounded between $0$ and $1$ by definition, the second term vanishes. Therefore
\[\lim_{n\to\infty}\rho_n(\pi(A\oplus B))=\lim_{n\to\infty} \rho_k(\pi(A\oplus B)\within\pi(O))=\rho(\pi(A\oplus B)\within\pi(O))=r\]
as desired.
\end{proof}

Using Theorem \ref{join}, we may take any set $X$ which is $r$-random and then $X\oplus X$ has intrinsic density $r$ but is not random as desired. However, this is merely a structural difference between the notion of randomness and the notion of intrinsic density as opposed to a computational difference.

\section{Core Theorems}

We would now like to build examples of sets with intrinsic density $r$ which cannot compute an $r$-random set. To do so, we shall prove two main theorems which allow us to apply the {\tt into} and {\tt within} operations to study asymptotic and intrinsic density.

\begin{theorem}
\label{theoremwithin}
Let $C$ be computable and $P(A)=r$. Then $P(A\within C)=r$.
\end{theorem}

\begin{proof}
Under the map which takes $c_n$ to $n$, $A\cap C$ is mapped to $A\within C$. However unless $C$ is $\omega$, this is not a permutation. Using Lemma \ref{permutationlemma}, we are able to massage this map into a permutation which takes $c_n$ to $n$ modulo a set of density $0$. Then under this permutation, $A\cap C$ (and $A$) goes to $A\within C$ modulo a set of density $0$. Therefore if $A\within C$ did not have intrinsic density $r$, $A$ could not either by Lemma \ref{densityzero}. 
\\
\\
Formally, assume $P(A\within C)\neq r$. Suppose $\pi$ is a computable permutation with $\overline{\rho}(\pi(A\within C))>r$. Let $f:C\to\omega$ be defined via $f(c_n)=n$. Then $f(A\cap C)=A\within C$:

\begin{center}
\begin{tikzcd}
A\cap C \arrow [r, "f"] & A\within C \arrow [r,"\pi"] & \pi(A\within C)
\end{tikzcd}
\end{center}

By Lemma \ref{permutationlemma}, there is $H\subseteq C$ computable with $\overline{\rho}(\pi(f(H)))=0$. Define $\pi_f:\omega\to\omega$ via $\pi_f(n)=f(n)$ for $n\in C\setminus H$, and for $n\in \overline{C}\sqcup H$ define $\pi_f(n)$ to be the least element of $f(H)$ not equal to $\pi_f(j)$ for some $j<n$. As $f$ agrees with $\pi_f$ on $C\setminus H$,
\[\pi_f((A\cap C)\setminus H)=f(A\cap C)\setminus f(H)=(A\within C)\setminus f(H)\]
Therefore by applying $\pi$,
\[\pi(\pi_f((A\cap C)\setminus H))=\pi((A\within C)\setminus f(H))=\pi(A\within C)\setminus \pi(f(H))\]
Using the above equality,
\[\overline{\rho}(\pi(\pi_f((A\cap C)\setminus H)))=\overline{\rho}(\pi(A\within C)\setminus\pi(f(H)))\]
As $\overline{\rho}(\pi(f(H)))=0$, we can apply Lemma \ref{densityzero} and see
\[\overline{\rho}(\pi(A\within C)\setminus\pi(f(H)))=\overline{\rho}(\pi(A\within C))\]
As $(A\cap C)\setminus H\subseteq A$,
\[\overline{\rho}(\pi(\pi_f(A)))\geq\overline{\rho}(\pi(\pi_f((A\cap C)\setminus H)))=\overline{\rho}(\pi(A\within C))\]
However, we assumed that $\overline{\rho}(\pi(A\within C))>r$, so $\overline{\rho}(\pi(\pi_f(A)))>r$. As $\pi\circ\pi_f$ is a computable permutation, this implies $P(A)\neq r$.
\\
\\
This proves that if $\pi$ is a computable permutation with $\overline{\rho}(\pi(A\within C))>r$, then $P(A)\neq r$. If there is no such permutation, there must be a computable permutation $\pi$ with $\underline{\rho}(\pi(A\within C))<r$ because we assumed that $P(A\within C)\neq r$. Then because 
\[(\pi(A\within C))\sqcup(\pi(\overline{A}\within C))=\pi((A\within C)\sqcup(\overline{A}\within C))=\pi(\omega)=\omega\]
we have $\rho_n(\pi(\overline{A}\within C))=1-\rho_n(\pi(A\within C))$ for all $n$. Therefore by the subtraction properties of the limit superior,
\[\overline{\rho}(\pi(\overline{A}\within C))\geq1-\underline{\rho}(\pi(A\within C))\]
As we assumed $\underline{\rho}(\pi(A\within C))<r$,
\[1-\underline{\rho}(\pi(A\within C))>1-r\]
Thus $\overline{\rho}(\pi(\overline{A}\within C))>1-r$. We now apply the previous case to get that $P(\overline{A})\neq1-r$, which automatically implies $P(A)\neq r$.
\end{proof}

If we are successful in building a set of intrinsic density $r$ which cannot compute an $r$-random set, then we will be able to use this theorem to automatically get many more such examples.
\\
\\
We now turn our attention {\tt into}, which is more useful. We first make a crucial observation about the asymptotic density of $B\into A$, which will be critical for investigating the intrinsic density of sets obtained via use of the {\tt into} operation.

\begin{lemma}
\label{densitylimit}
\mbox{}
\begin{itemize}
    \item $\overline{\rho}(B\into A)\leq \overline{\rho}(B)\overline{\rho}(A)$.
    \item $\underline{\rho}(B\into A)\geq \underline{\rho}(B)\underline{\rho}(A)$.
\end{itemize}
\end{lemma}

\begin{proof}
By Lemma \ref{technicallimit},
\[\overline{\rho}(B\into A)=\limsup_{n\to\infty}\frac{n}{a_{b_n}}=\limsup_{n\to\infty}\frac{n}{a_{b_n}}\cdot 1=\limsup_{n\to\infty}\frac{n}{a_{b_n}}\cdot\frac{b_n}{b_n}\]
By the submultiplicativity of the limit superior, 
\[\overline{\rho}(B\into A)\leq (\limsup_{n\to\infty}\frac{b_n}{a_{b_n}})(\limsup_{n\to\infty}\frac{n}{b_n})=(\limsup_{n\to\infty}\frac{b_n}{a_{b_n}})\overline{\rho}(B)\]
Now $\{\frac{b_n}{a_{b_n}}\}_{n\in\omega}$ is a subsequence of $\{\frac{n}{a_n}\}_{n\in\omega}$, so
\[\limsup_{n\to\infty}\frac{b_n}{a_{b_n}}\leq\limsup_{n\to\infty}\frac{n}{a_n}=\overline{\rho}(A)\]
Therefore $\overline{\rho}(B\into A)\leq \overline{\rho}(B)\overline{\rho}(A)$ as desired.
\\
\\
The case for the limit inferior is nearly identical, reversing $\leq$ to $\geq$ and using supermultiplicativity along with the corresponding identity from Lemma \ref{technicallimit}.
\end{proof}

\begin{corollary}
If $\rho(A)=\alpha$ and $\rho(B)=\beta$, then $\rho(B\into A)=\alpha\beta$.
\end{corollary}

Therefore, if $B\into A$ has intrinsic density, its intrinsic density must be the product of the densities of $A$ and $B$. This will occur in certain instances, which is our second main theorem. Recall that a set $X$ has $Y$-intrinsic density, or intrinsic density relative to $Y$, if its density is invariant under all $Y$-computable permutations as opposed to just the computable ones. We use $P_Y(X)$ to denote the $Y$-intrinsic density of $X$ if it exists.

\begin{theorem}
\label{core}
If $P(A)=\alpha$ and $P_A(B)=\beta$, then $P(B\into A)=\alpha\beta$.
\end{theorem}

\begin{proof}
The proof is similar to the proof of Theorem \ref{theoremwithin}, however we shall present it fully here without referring to techniques from that proof, as it is quite technical. The idea is that for any fixed computable permutation $\pi$, there is an $A$-computable permutation which sends $B$ to $\pi(B\into A)\within\pi(A)$  modulo a set of density $0$. Therefore if $\pi$ witnesses that $B\into A$ does not have intrinsic density $\alpha\beta$, i.e.\ $\pi(B\into A)$ does not have density $\alpha\beta$, and then as $A$ has intrinsic density $\alpha$ by assumption, Lemma \ref{densitylimit} will show that $\pi(B\into A)\within\pi(A)$ does not have density $\beta$. Thus $B$ does not have $A$-intrinsic density $\beta$.
\\
\\
Formally, assume $P(A)=\alpha$. Assume that $P(B\into A)\neq \alpha\beta$. We shall show that $P_A(B)\neq \beta$. First suppose that there is some computable permutation $\pi$ such that $\overline{\rho}(\pi(B\into A))>\alpha\beta$. We shall let $\pi(A)=\{p_0<p_1<p_2<\dots\}$. Let $f:A\to\omega$ be defined via $f(a_n)=n$ and $g:\pi(A)\to\omega$ via $g(p_n)=n$,  $f$ maps $A$ to its indices and $g$ maps $\pi(A)$ to its indices. Then $f(B\into A)=B$ and $g(\pi(B\into A))=\pi(B\into A)\within\pi(A)$:

\begin{center}
\begin{tikzcd}
B\into A \arrow[r, "\pi"] \arrow [d, "f"] & \pi(B\into A) \arrow[d, "g"] \\
B & \pi(B\into A)\within\pi(A)
\end{tikzcd}
\end{center}

Note by Lemma \ref{densitylimit} that $\overline{\rho}(\pi(B\into A)\within\pi(A))>\beta$: From the definition,
\[(\pi(B\into A)\within\pi(A))\into \pi(A))=\pi(B\into A)\]
and $\overline{\rho}(B\into A)>\alpha\beta$ by assumption. $\overline{\rho}(\pi(A))=\alpha$ because $P(A)=\alpha$, so $\overline{\rho}(\pi(B\into A)\within\pi(A))\leq\beta$ would contradict Lemma \ref{densitylimit}.
\\
\\
From this point forward we shall let
\[X=\pi(B\into A)\within\pi(A)\]
for the sake of readability.
\\
\\
By Lemma \ref{permutationlemma} relativized to $A$ and applied to $g\circ\pi$, there is an $A$-computable set $H\subseteq A$ such that:
\[\overline{\rho}(g(\pi(H)))=0\]

We shall now define permutations which preserve the properties of $f$ and $g$ outside of $H$. Define $\pi_f:\omega\to\omega$ via $\pi_f(k)=f(k)$ for $k\in A\setminus H$, and for $k\in \overline{A}\sqcup H$, let $\pi_f(k)$ be the least element of $f(H)$ not equal to $\pi_f(m)$ for some $m<k$. Define $\pi_g:\omega\to\omega$ similarly using $\pi(A)$, $\pi(H)$, and $g(\pi(H))$ in place of $A$, $H$, and $f(H)$ respectively. Then $\pi_f$ and $\pi_g$ are $A$-computable because $H$, $f$, and $g$ are, and it is a permutation because $f$ and $g$ are bijections (from $A$ and $\pi(A)$ to $\omega$ respectively) which have been modified to be total without violating injectivity or surjectivity.
\\
\\
Now we shall compute $\pi_g(\pi(\pi_f^{-1}(B\setminus f(H))))$. As $f(B\into A)=B$ and $f$ agrees with $\pi_f$ on $\overline{H}$,
\[\pi_f^{-1}(B\setminus f(H))=(B\into A)\setminus H\]
Furthermore
\[\pi((B\into A)\setminus H)=\pi(B\into A)\setminus\pi(H)\]
As $g(\pi(B\into A))=X$ and $\pi_g$ agrees with $g$ on $\overline{\pi(H)}$, 
\[\pi_g(\pi(B\into A)\setminus\pi(H))=g(\pi(B\into A))\setminus g(\pi(H))=X\setminus g(\pi(H))\]
Thus $\pi_g(\pi(\pi_f^{-1}(B\setminus f(H))))=X\setminus g(\pi(H))$. As $\overline{\rho}(g(\pi(H))=0$, Lemma \ref{densityzero} shows
\[\overline{\rho}(X\setminus g(\pi(H)))=\overline{\rho}(X)\]
By the definition of $X$,
\[\overline{\rho}(X)=\overline{\rho}(\pi(B\into A)\within\pi(A))\]
which is greater than $\beta$ by the above. As $B\setminus f(H)\subseteq B$, 
\[\pi_g(\pi(\pi_f^{-1}(B\setminus f(H))))\subseteq\pi_g(\pi(\pi_f^{-1}(B)))\]
and thus
\[\overline{\rho}(\pi_g(\pi(\pi_f^{-1}(B))))\geq \overline{\rho}(\pi_g(\pi(\pi_f^{-1}(B\setminus f(H)))))\]
Therefore
\[\overline{\rho}(\pi_g(\pi(\pi_f^{-1}(B))))\geq \overline{\rho}(\pi(B\into A)\within\pi(A))>\beta\]
As $\pi_g\circ\pi\circ\pi_f^{-1}$ is an $A$-computable permutation, $P_A(B)\neq\beta$.
\\
\\
Therefore we have proved that if there is some computable permutation $\pi$ such that $\overline{\rho}(\pi(B\into A))>\alpha\beta$, then $P_A(B)\neq \beta$. If there is no such permutation, then there must be a computable permutation $\pi$ such that $\underline{\rho}(\pi(B\into A))<\alpha\beta$ because we assumed $P(B\into A)\neq\alpha\beta$. As $A=(B\into A)\sqcup(\overline{B}\into A)$, $\pi(A)=\pi(B\into A)\sqcup\pi(\overline{B}\into A)$. Therefore
\[\overline{\rho}(\pi(\overline{B}\into A))=\overline{\rho}(\pi(A)\setminus\pi(B\into A))\]
The fact that $\rho_n(\pi(A))=\rho_n(\pi(B\into A))+\rho_n(\pi(A)\setminus\pi(B\into A))$ combined with the properties of the limit superior with regards to subtraction implies
\[\overline{\rho}(\pi(A)\setminus\pi(B\into A))\geq\overline{\rho}(\pi(A))-\underline{\rho}(\pi(B\into A))\]
We know that $\overline{\rho}(\pi(A))=\alpha$ because $P(A)=\alpha$. As we assumed that $\underline{\rho}(\pi(B\into A))<\alpha\beta$,
\[\overline{\rho}(\pi(\alpha))-\underline{\rho}(\pi(B\into A))>\alpha-\alpha\beta=\alpha(1-\beta)\]
Bringing this together,
\[\overline{\rho}(\pi(\overline{B}\into A))>\alpha(1-\beta)\]
Thus we can apply the first case of the proof to show that $P_A(\overline{B})\neq 1-\beta$, which automatically implies $P_A(B)\neq\beta$, so we are done.
\end{proof}

\begin{corollary}
If $P(A)=\alpha$ and $P_A(B)=\beta$, then $P(A\cap B)=\alpha\beta$.\footnote{Astor \cite{intrinsicdensity} proved this for the special case when $A$ has intrinsic density $\alpha$ and $B$ is 1-random relative to $A$.}
\end{corollary}

\begin{proof}
By definition,
\[A\cap B=(B\within A)\into A\]
As $P_A(B)=\beta$, Theorem \ref{theoremwithin} relativized to $A$ shows that $P_A(B\within A)=\beta$. Therefore we can apply Theorem \ref{core} to $A$ and $B\within A$ to get that 
\[P((B\within A)\into A)=P(A\cap B)=\alpha\beta\]
\end{proof}

It is natural to ask if any of the intrinsic density requirements in Theorem \ref{core} can be dropped. It is immediate that we cannot drop the requirement that $A$ has intrinsic density: $P_A(\omega)=1$ for any $A$, so $\omega$ always satisfies the requirements on $B$, but $\omega\into A=A$, so $A$ must have intrinsic density. Similarly, $B\into\omega=B$ for any $B$, so $B$ must have intrinsic density. Therefore the only possible weakening of Theorem \ref{core} would be to require $P(B)=\beta$ as opposed to $P_A(B)=\beta$. However, this fails.

\begin{proposition}
\label{weakening}
Let $P(A)=\frac{1}{2}$. Then $P(A\oplus\overline{A})=\frac{1}{2}$ but $A\into(A\oplus\overline{A})$ does not have intrinsic density.
\end{proposition}

\begin{proof}
Note that $A\oplus\overline{A}$ has intrinsic density $\frac{1}{2}$ by Theorem \ref{join} as $P(A)=\frac{1}{2}$ implies $P(\overline{A})=\frac{1}{2}$.
\\
\\
Let $E$ represent the set of even numbers. Notice that $A\oplus \overline{A}$ contains exactly one of $2k$ or $2k+1$ for all $k\in\omega$. Therefore the $n$-th element of $A\oplus \overline{A}$ is $2n$ if $n\in A$ and $2n+1$ if $n\not\in A$. Thus
\[E\within (A\oplus\overline{A})=A\]
by definition. By the properties of the {\tt within} operation,
\[A\into(A\oplus \overline{A})=(E\within(A\oplus\overline{A}))\into (A\oplus\overline{A})=E\cap(A\oplus\overline{A})=A\oplus\emptyset\]
By Proposition \ref{join}, however, $A\oplus\emptyset$ does not have intrinsic density.
\end{proof}

\section{Intrinsic Density $r$ sets which cannot compute an $r$-random}
We are now ready to construct sets of intrinsic density $r$ which cannot compute an $r$-random set. To do this, we would like to have a countable collection of sets which all have intrinsic density relative to each other so that we may apply Theorem \ref{core} repeatedly. 

\begin{lemma}
\label{randomsequence}
Any 1-random set $X$ uniformly computes a countable, disjoint sequence of sets $\{A_i\}_{i\in\omega}$ such that $A_i$ has intrinsic density $\frac{1}{2^{i+1}}$, i.e. $P(A_i)=\frac{1}{2^{i+1}}$, for each $i$. Furthermore, the $A_i$'s form a partition of $\omega$.
\end{lemma}

\begin{proof}
    Recall that given a set $X$, $X^{[i]}$ denotes the $i$-th column of $X$, i.e.\ $\{n:\langle i,n\rangle\in X\}$. Let $X\subseteq\omega$ be 1-random. Then for all $i$, $X^{[i]}$ is 1-random relative to $\bigoplus_{j\neq i} X^{[j]}$. (Essentially Van Lambalgen \cite{vanlambalgen}, Downey-Hirschfeldt \cite[Corollary 6.9.6]{downeyhirschfeldt}) Proposition \ref{randomdensity} relativizes to the fact that $Z$-1-randoms have $Z$-intrinsic density $\frac{1}{2}$. In particular, taking a single 1-random automatically gives us infinitely many mutually 1-random sets, and thus infinitely many sets with intrinsic density $\frac{1}{2}$ relative to each other. Using these together with Theorem \ref{core}, we can construct the desired sequence, where the mutual randomness ensures that the conditions of Theorem \ref{core} are met.
    \\
    \\
    Let $B_0=\omega$. Given $B_n$, let 
    \[A_{n}=\overline{X^{[n]}}\into B_n\]
    and 
    \[B_{n+1}=X^{[n]}\into B_n\]
    Note that for all $i$, $B_{i+1}\subseteq B_i$ and $A_{i}\cap B_{i+1}=\emptyset$, as $B_{i+1}=X^{[i]}\into B_i$ and $A_{i}=\overline{X^{[i]}}\into B_{i}$. Then for $i<j$, $A_i\cap A_j=\emptyset$ because $A_j\subseteq B_j\subseteq B_{i+1}$. Thus $\{A_i\}_{i\in\omega}$ is disjoint. Furthermore, the $A_i$'s and $B_i$'s are in fact uniformly computable in $X$, as $X$ can uniformly compute all of its columns and $X\into Y$ is uniformly computable in $X\oplus Y$. We now verify that $P(A_i)=\frac{1}{2^{i+1}}$ and $P(B_i)=\frac{1}{2^i}$ by induction.
    \\
    \\
    $P(B_0)=P(\omega)=1$, and $B_0$ is computable. Suppose that $B_i$ is $\bigoplus_{j<i} X^{[j]}$-computable and that $P(B_i)=\frac{1}{2^i}$. Then $B_{i+1}=X^{[i]}\into B_i$ and $A_i=\overline{X^{[i]}}\into B_i$ are both $B_i\oplus X^{[i]}$-computable, and therefore $\bigoplus_{j< i+1} X^{[j]}$-computable. By the above, both $X^{[i]}$ and $\overline{X^{[i]}}$ are 1-random relative to $B_i$ and thus have intrinsic density $\frac{1}{2}$ relative to $B_i$. Therefore $P_{B_i}(X^{[i]})=P_{B_i}(\overline{X^{[i]}})=\frac{1}{2}$ by the relativization of Proposition \ref{randomdensity}. Thus by Theorem \ref{core} and the induction hypothesis, \[P(A_i)=P(\overline{X^{[i]}}\into B_i)=P(\overline{X^{[i]}})P(B_i)=\frac{1}{2}\cdot\frac{1}{2^i}=\frac{1}{2^{i+1}}\]
    A nearly identical argument for $P(B_{i+1})$ verifies $P(B_{i+1})=\frac{1}{2^{i+1}}$, which completes the induction.
    \\
    \\
    Finally, suppose for the sake of contradiction that the $A_i$'s do not form a partition of $\omega$. Since we have already shown that the sequence is disjoint, there must exist a least $m$ with $m\not\in A_i$ for any $i$. Therefore, there must be some $k$ such that $m$ is the least element of $B_n$ for all $n\geq k$. This is because every $m$ is in $B_0$ and $B_{i+1}$ is a subset of $B_i$ missing only elements of $A_i$. It follows that $0\in X^{[n]}$ for all $n>k$, as $0\in\overline{X^{[n]}}$ would imply that $m\in A_n$ since $A_n=\overline{X^{[n]}}\into B_n$ and $m$ is the least, i.e. $0$-th, element of $B_n$. However, this means that $\{\langle n,0\rangle:n>k\}$ is an infinite computable subset of $X$, which contradicts the assumption that $X$ is 1-random since it is a basic fact that 1-random sets cannot have infinite computable subsets. Therefore, every $m$ must be in some $A_i$ as desired.
\end{proof}

Jockusch and Schupp \cite{JockuschSchupp} proved that asymptotic density enjoys a restricted form of countable additivity: if there is a countable sequence $\{S_i\}_{i\in\omega}$ of disjoint sets such that $\rho(S_i)$ exists for all $i$ and 
\[\lim_{n\to\infty} \overline{\rho}(\bigsqcup_{i>n} S_i)=0\]
then
\[\rho(\bigsqcup_{i\in\omega}S_i)=\mathop{\Sigma}_{i=0}^\infty \rho(S_i)\]
The intrinsic density analog of this results follows immediately from the fact that permutations preserve disjoint unions. That is, if there is a countable sequence $\{S_i\}_{i\in\omega}$ of disjoint sets such that $P(S_i)$ exists for all $i$ and 
\[\lim_{n\to\infty} \overline{P}(\bigsqcup_{i>n} S_i)=0\]
then
\[P(\bigsqcup_{i\in\omega}S_i)=\mathop{\Sigma}_{i=0}^\infty P(S_i)\]

Note that $\lim_{n\to\infty}\overline{P}(\bigsqcup_{i>n} A_i)=0$ must be true for any sequence satisfying Lemma \ref{randomsequence}, as $\lim_{n\to\infty} P(\bigsqcup_{i\leq n} A_i)=1$. This together with the previous lemma allows us to construct our desired set.

\begin{theorem}
\label{every}
For every $r\in(0,1)$ and any 1-random set $X$, $r\oplus X$ computes a set with intrinsic density $r$.
\end{theorem}

\begin{proof}
Let $r\in(0,1)$. Let $B_r\subseteq \omega$ be the set whose characteristic function is identified with the binary expansion that gives $r$,  the set of all $n$ such that the $n$-th bit in the binary expansion for $r$ is 1. Let $\{A_i\}_{i\in\omega}$ be as in Lemma \ref{randomsequence} applied to $X$, i.e. a uniformly $X$-computable partition of $\omega$ with $A_i$ having intrinsic density $\frac{1}{2^{i+1}}$ for each $i$. Let $X_r=\bigsqcup_{n\in B_r} A_n$. We now describe the process by which $X_r$ is computable from $r\oplus X$. For a given $m$, $m\in A_n$ for some $n$ since the $A_i$'s form a partition of $\omega$. $X$ can uniformly compute the $A_i$'s and thus compute $n$ where $m\in A_n$. Then $m\in X_r$ if and only if $n\in B_r$, which $r$ can compute.
\\
\\
Now, note that 
\[\lim_{n\to\infty}\overline{P}(\bigsqcup_{i\in B_r, i>n} A_i)=0\]
because $\bigsqcup_{i\in B_r, i>n} A_i\subseteq \bigsqcup_{i>n} A_i$ and $\lim_{n\to\infty}\overline{P}(\bigsqcup_{i>n} A_i)=0$. By the fact that countable unions sum intrinsic densities and the definition of $X_r$,
\[P(X_r)=\mathop{\Sigma}_{n\in B_r} P(A_n)=\mathop{\Sigma}_{n\in B_r} \frac{1}{2^{n+1}}\]
By the definition of the binary expansion,
\[P(X_r)=\mathop{\Sigma}_{n\in B_r} \frac{1}{2^{n+1}}=r\]
\end{proof}

\begin{proposition}[Reimann and Slaman \cite{ReimannSlaman}]
\label{representation}
No $r$-random set $X$ can be computable from $r$.
\end{proposition}

\begin{proof}
Reimann and Slaman \cite[Proposition 2.3]{ReimannSlaman} says that if $Y\subseteq\omega$ is a representation of some measure $\mu$ on $2^\omega$, then $Y$ computes a function $g_\mu:2^{<\omega}\times\omega\to\mathbb{Q}$ such that for all $\sigma\in2^{<\omega}$ and all $n\in\omega$, \[|g_\mu(\sigma,n)-\mu(\sigma)|\leq 2^{-n}\]
Take $\mu$ to be $\mu_r$, the $r$-Bernoulli measure. Fix $\sigma_1$ to be $1$, the binary string of length one with first bit $1$. Then we have $\mu_r(\sigma_1)=r$. Thus the $Y$-computable function $g(n)=g_{\mu_r}(\sigma_1,n)$ has the property that for all $n\in\omega$,
\[|g_{\mu_r}(\sigma_1,n)-\mu_r(\sigma_1)|=|g(n)-r|\leq 2^{-n}\]
In other words, $g$ can be used to compute $r$. As $Y$ was arbitrary, every representation of $\mu_r$ computes $r$.
\\
\\
By definition (\cite[Definition 3.2]{ReimannSlaman}), a set $X$ is $r$-random if and only if it is random with respect to some representation $Y$ of $\mu_r$. Therefore $Y$ cannot compute $X$. However, by the above argument, $Y$ computes $r$. Thus $r$ cannot compute $X$.
\end{proof}

\begin{theorem}
For almost all $r$, $r$ computes a set $A$ which has intrinsic density $r$ but cannot compute an $r$-random set.
\end{theorem}

\begin{proof}
Let $r$ be 1-random. Then we may apply Theorem \ref{every} with $r$ playing the role of both $r$ and $X$ to obtain an $r$-computable set $A$ which has intrinsic density $r$. Every $A$-computable set is $r$-computable because $A$ is $r$-computable, but by Proposition \ref{representation}, no $r$-random set can be $r$-computable. Therefore, $A$ cannot compute an $r$-random set.
\\
\\
As almost all reals are 1-random, this completes the proof.
\end{proof}

From each such set, we may use Theorem \ref{theoremwithin} to generate more examples.
 
\section{Closing Remarks and Future Work}

We showed a computational separation between the notions of intrinsic density and randomness by constructing examples of sets with intrinsic density $r$ which cannot compute $r$-random sets. There is room for future work in investigating whether this can be done for all $r$, or at least for all noncomputable $r$.
\\
\\
Additionally, there is room for applying these techniques to compare other notions of stochasticity with randomness. Many of the results proved for intrinsic density have analogs for MWC stochasticity, however it is currently open whether infinite unions work in that setting as they do for intrinsic density. Thus the analog of Theorem \ref{every} for MWC stochasticity remains open.


\begin{thebibliography}{}

\bibitem{Ambosspies}
Klaus Ambos-Spies.
Algorithmic randomness revisited.
\textit{Language, Logic and Formalization of Knowledge},
Coimbra Lecture and Proceedings of a Symposium held in Siena in.
1997.

\bibitem{intrinsicdensity}
Eric P. Astor.
Asymptotic density, immunity, and randomness.
\textit{Computability},
4(2):141-158, 2015.
ISSN 2211-3568.
doi: 10.3233/COM-150040.

\bibitem{downeyhirschfeldt}
Rodney G. Downey and Denis R. Hirschfeldt.
Algorithmic randomness and complexity.
\textit{Theory and Applications of Computability},
Springer,
2010.

\bibitem{JockuschSchupp}
Carl G. Jockusch, Jr. and Paul E. Schupp.
Generic computability, Turing degrees, and asymptotic density.
\textit{Journal of the London Mathematical Society},
85(2):472–490, 2012.
ISSN 0024-6107.
doi.org/10.1112/jlms/jdr051.

\bibitem{Nies}
Andr\'e Nies.
Computability and randomness.
\textit{Oxford Logic Guides},
Oxford University Press,
2009.

\bibitem{ReimannSlaman}
Jan Reimann and Theodore A. Slaman.
Measures and their random reals.
\textit{ArXiv},
arXiv:0802.2705, 2013.
v2

\bibitem{vanlambalgen}
M. van Lambalgen.
The axiomatization of randomness.
\textit{The Journal of Symbolic Logic},
55(3):1143-1167, 1990.
ISSN: 0022-4812
doi: 10.2307/2274480 

\bibitem{wang}
Yongge Wang.
Randomness and complexity.
PhD dissertation,
Ruprecht-Karls-Universit\"at Heidelberg,
August 1996.
https://webpages.uncc.edu/yonwang/papers/thesis.pdf

\end{thebibliography}
\end{document}